\documentclass[a4paper]{article}

\usepackage{amsfonts,amssymb,mathrsfs,amsmath,amsthm}
\usepackage{mathtools}
\usepackage{hyperref}
\usepackage{enumitem}
\usepackage{csquotes}
\usepackage{tikz-cd}
\usetikzlibrary{arrows,calc,matrix}
\usepackage{adjustbox}
\usepackage[all]{xy}
\usepackage[a4paper,hmargin=2.5cm,vmargin=2.5cm]{geometry}
\usepackage{microtype}

\usepackage[backend=biber,date=short,bibencoding=utf8,sorting=nyt,style=numeric]{biblatex}
\bibliography{isotropic_Grassmanians}

\newcommand{\Sp}{\mathrm{Sp}}
\newcommand{\Gr}{\mathrm{Gr}}
\newcommand{\IGr}{\mathrm{IGr}}
\newcommand{\IFl}{\mathrm{IFl}}
\newcommand{\GL}{\mathrm{GL}}
\newcommand{\Fl}{\mathrm{Fl}}
\newcommand{\Db}{{\mathrm{D}^b}}
\newcommand{\cE}{\mathcal{E}}
\newcommand{\cF}{\mathcal{F}}
\newcommand{\cH}{\mathcal{H}}
\newcommand{\cK}{\mathcal{K}}
\newcommand{\cU}{\mathcal{U}}
\newcommand{\cS}{\mathcal{S}}
\newcommand{\Sy}{\mathrm{S}}
\newcommand{\cO}{\mathcal{O}}
\newcommand{\vn}{\mathbf{d}}

\newtheorem{theorem}{Theorem}[section]

\newtheorem{proposition}[theorem]{Proposition}
\newtheorem{lemma}[theorem]{Lemma}

\newtheorem{corollary}[theorem]{Corollary}

\theoremstyle{remark}

\newtheorem{remark}[theorem]{Remark}

\theoremstyle{definition}

\begin{document}
	
\title{Secondary staircase complexes on isotropic Grassmannians}
\author{Alexander~Novikov
\thanks{The study has been funded within the framework of the HSE University Basic Research Program.}}
\date{}

\setlength{\headheight}{15pt}

\maketitle{}

\begin{abstract} 
We introduce a class of equivariant vector bundles on isotropic symplectic Grassmannians~\(\mathrm{IGr}(k,2n)\)
defined as appropriate truncations of staircase complexes
and show that these bundles can be assembled into a number of complexes 
quasi-isomorphic to the symplectic wedge powers of the symplectic bundle on~\(\mathrm{IGr}(k,2n)\).
We are planning to use these secondary staircase complexes to study fullness of exceptional collections 
in the derived categories of isotropic Grassmannians and Lefschetz exceptional collections on~\(\mathrm{IGr}(3,2n)\).
\end{abstract}

\section{Introduction}

The bounded derived category of coherent sheaves~\(\Db(X)\) 
is an important invariant of an algebraic variety~\(X\), but its structure can be complicated. 
The simplest case is when~\(\Db(X)\) has a full exceptional collection~\((E_{1}, E_{2}, \ldots, E_{m})\). 
Then each object of~\(\Db(X)\) admits a unique filtration, with the \(i\)-th subquotient being a direct sum of shifts of the objects~\(E_i\). 
Therefore, an exceptional collection serves as a kind of non-commutative basis for~\(\Db(X)\). 
A long-standing conjecture predicts the existence of a full exceptional collection
on all projective homogeneous varieties of reductive algebraic groups, see~\cite{KuzPol}.

The first example of a full exceptional collection was constructed by Beilinson, 
who showed in~\cite{Beilinson1978} that the twists of the structure sheaf~\(\mathcal{O},\mathcal{O}(1), \ldots, \mathcal{O}(n)\) 
form such a collection on the projective space~\(\mathbb{P}^n\). 
Afterwards, Kapranov~\cite{kap84} constructed full exceptional collections 
on the Grassmannians and flag varieties of groups \(\GL_n\) and on smooth quadrics. 
Later, Fonarev constructed in~\cite{Fonarev_2013,F15} other exceptional collections on the Grassmannians
and suggested a new proof of fullness based on a consistent application of the so-called \emph{staircase complexes}.

For symplectic isotropic Grassmannians~\(\IGr(k,2n)\), the progress has been slower.
As the case~\(\IGr(1, 2n) = \mathbb{P}^{2n - 1}\) is covered by the result of Beilinson, 
the first new case is the isotropic Grassmannian of lines~\(\IGr(2, 2n)\);
in this case a full exceptional collection was constructed in~\cite{lines}. 
Later, Kuznetsov and Polischchuk~\cite{KuzPol}, extending the ideas from~\cite{LG3} and~\cite{LG45}, 
where the cases of~\(\IGr(3,6)\), \(\IGr(4, 8)\), and~\(\IGr(5, 10)\) were discussed,
developed a general approach to constructing exceptional collections of expected length on all Grassmannians of classical groups. 
However, fullness of these collections was proved only in the Lagrangian case~\(\IGr(n,2n)\) by Fonarev~\cite{F19}. 
Finally, full exceptional collections on~\(\IGr(3,8)\) and~\(\IGr(3,10)\) 
were constructed in~\cite{38} and~\cite{N20} by Guseva and the author, respectively.

The above papers conveyed an important message:
to construct sufficiently long exceptional collections 
one has to consider equivariant vector bundles which are \emph{not necessarily irreducible},
and to prove fullness of the constructed collections,
one has to find exact sequences relating these bundles, 
similar to the staircase complexes that work so well in the case of~\(\GL_n\).

Some complexes of this sort already appeared in the papers mentioned above.
For instance, the key step in the proof~\cite{lines} of fullness of the exceptional collection on~\(\IGr(2,2n)\) 
was a construction of a certain bicomplex in~\cite[Proposition~5.3]{lines},
that can be considered as a complex consisting of the objects represented by its rows;
one of our results is another construction of such a complex 
fixing an inaccuracy made in~\cite{lines}, see Remarks~\ref{rem:bicomplex},~\ref{rem:bicomplex2} and Corollary~\ref{cor:bicomplex-maps}.
Similarly, the proof of fullness of an exceptional collection on~\(\IGr(3,8)\) and~\(\IGr(3,10)\)
uses bicomplexes described in~\cite[Section~5.2]{38} and~\cite[Section~3.2]{N20}, respectively.

The main result of this paper is a large generalization of the above constructions.
We describe a class of (non-irreducible) \(\Sp_{2n}\)-equivariant vector bundles~\(\cK_{\IGr(k,2n)}^{\alpha_1,\alpha_2}\) on~\(\IGr(k,2n)\)
defined as appropriate truncations of staircase complexes, 
show that they can be assembled into natural complexes,
and identify their cohomology sheaves.

To state our results we need to introduce some notation.
Let~\(V\) be a vector space of dimension~\(2n\) and consider the Grassmannian~\(\Gr(k,V) = \Gr(k,2n)\).
We also fix a symplectic form~\(\omega\) on~\(V\) and for~\(2 \le k \le n\) 
consider the isotropic Grassmannian~\(\IGr(k,V) = \IGr(k,2n)\).
We denote by~\(\cU_k\) the tautological subbundle of rank~\(k\) in~\(V \otimes \cO\) on~\(\Gr(k,V)\)
and for a dominant weight~\((\alpha_1,\alpha_2,\dots,\alpha_k)\) of~\(\GL_k\) 
we denote by~\(\Sigma^{\alpha_1,\alpha_2,\dots,\alpha_k}\cU_k^\vee\) 
the result of application of the corresponding Schur functor to~\(\cU_k^\vee\), the dual tautological bundle. In particular,~\(\Sigma^{a,0,\dots,0}\cU_k^\vee \cong \mathrm{S}^a \cU_k^\vee\). 
 
For arbitrary integers~\(2n - k \ge \alpha_1 \ge \alpha_2 \ge 0\) and~\(2 \le k \le n\)
we define a \(\GL_{2n}\)-equivariant vector bundle~\(\cK_{\Gr(k,V)}^{\alpha_1,\alpha_2}\) on~\(\Gr(k,V)\) as a special case of the duals to the bundles~\(\cE^{\lambda,\mu}\) defined in~\cite[Section~3]{F15}:
\begin{equation}\label{eq:K=E*}
	\cK^{\alpha} \coloneqq \left(\cE^{(\alpha_2),(1^{\alpha_1})}\right)^\vee, \qquad 
	\text{where~\((1^{\alpha_1}) = (\underbrace{1, \dots, 1}_{\alpha_1})\).}
\end{equation}

Using the exact sequences from~\cite[Theorem~4.3]{F15} and isomorphisms~\(\cE^{\lambda,0} \cong \Sigma^\lambda \cU_k\), we get the following \(GL_{2n}\)-equivariant resolution, see~\eqref{eq:staircase-3-lines}:
\begin{multline}
	\label{eq:def-ckk}
	0 \to 
	\cK_{\Gr(k,V)}^{\alpha_1,\alpha_2} 
	\to \wedge^{\alpha_1+1} V^\vee \otimes \Sigma^{\alpha_2-1,0,\dots,0}\cU_k^\vee \to 
	\dots \to 
	\wedge^{\alpha_1-\alpha_2+2} V^\vee \otimes \Sigma^{\alpha_2-1,\alpha_2-1,0,\dots,0}\cU_k^\vee \\ \to 
	\wedge^{\alpha_1-\alpha_2} V^\vee \otimes \Sigma^{\alpha_2,\alpha_2,0,\dots,0}\cU_k^\vee \to 
	\dots \to
	\Sigma^{\alpha_1,\alpha_2,0,\dots,0}\cU_k^\vee \to 0.
\end{multline}
We furthermore define~\(\cK_{\IGr(k,V)}^{\alpha_1,\alpha_2} \coloneqq \cK_{\Gr(k,V)}^{\alpha_1,\alpha_2}\vert_{\IGr(k,V)}\).

We also consider the tautological bundle~\(\cU_k^\perp\) of rank~\(2n - k\) defined by the exact sequence
\begin{equation}
\label{eq:def-uperp}
0 \to \cU_k^\perp \to V^\vee \otimes \cO \to \cU_k^\vee \to 0.
\end{equation}
In particular, if~\(\alpha_2 = 0\), then there is no second line in~\eqref{eq:def-ckk}, and the resolution is
\begin{equation*}
	0 \to \cK_{\Gr(k,V)}^{\alpha_1, 0} \to 
	\wedge^{\alpha_1} V^\vee \otimes \cO \to 
	\wedge^{\alpha_1-1} V^\vee \otimes \cU_k^\vee \to 
	\dots \to
	\mathrm {S}^{\alpha_1}\cU_k^\vee \to 0.
\end{equation*}
So, we get~\(\cK_{\Gr(k,V)}^{\alpha_1, 0} \cong \Lambda^{\alpha_1} \cU_k^{\perp}\).

The symplectic form~\(\omega\) induces, after restriction of~\(\cU_k\) and~\(\cU_k^\perp\) to~\(\IGr(k,V)\), 
the natural embedding~\(\cU_k \hookrightarrow \cU_k^\perp\),
and we denote the quotient bundle by~\(\cS_k\), so that we have an exact sequence
\begin{equation}
\label{eq:def-csk}
0 \to \cU_k \xrightarrow{\ \omega\ } \cU_k^\perp \xrightarrow{\quad} \cS_k \to 0.
\end{equation}
Finally, note that the symplectic form~\(\omega\) induces a symplectic structure on~\(\cS_k\), which we denote by~\(\omega_\cS\),
and allows us to define for~\(0 \le i \le n - k = \tfrac12\operatorname{rank}(\cS_k)\) the symplectic wedge powers 
\begin{equation}
\label{eq:def-wedge-sp}
\wedge^i_\Sp\cS_k \coloneqq \operatorname{Coker}(\wedge^{i-2}\cS_k \xrightarrow{\ \omega_\cS\ } \wedge^i\cS_k).
\end{equation}

\begin{theorem}
\label{thm:main}
For any~\(0 \le t \le 2n - k\) there is an \(\Sp(V)\)-equivariant complex of vector bundles
\begin{equation}
\label{eq:ckk-special}
\cK_{t}^\bullet \coloneqq
\Big\{
\cK_{\IGr(k,2n)}^{2n-k-t,0} \to 
\cK_{\IGr(k,2n)}^{2n-k+1-t,1} \to 
\dots \to
\cK_{\IGr(k,2n)}^{2n-k-1,t-1} \to 
\cK_{\IGr(k,2n)}^{2n-k,t} 
\Big\},
\end{equation}
where the leftmost term is in degree~\(0\) such that its cohomology sheaves are
\begin{equation}
\label{eq:ch-cckk}
\cH^i(\cK_t^\bullet) \cong
\begin{cases}
\wedge_\Sp^{t}\cS_k(-1), & \text{if~\(i = 0\) and~\(0 \le t \le n-k\)},\\
\wedge_\Sp^{2(n-k+1) - t}\cS_k(-1), & \text{if~\(i = 1\) and~\(n-k+2 \le t \le 2(n-k+1)\)},\\
0, & \text{otherwise}.
\end{cases}
\end{equation}
In particular, the complex~\(\cK_{t}^\bullet\) is acyclic if~\(t = n - 1\) or~\(2n - 2k + 3 \le t \le 2n - k\). 
\end{theorem}

In our future work we are planning to use these results, to study fullness 
of the exceptional collection on~\(\IGr(k,2n)\) constructed in~\cite{KuzPol} for any~\(n\) and~\(2 \le k \le n\), and minimal Lefschetz exceptional collections on~\(\IGr(3, 2n)\) for any~\(n\).

To prove Theorem~\ref{thm:main} we first consider in detail the case~\(k = 2\).
In this case the bundles~\(\cK_{\Gr(2,V)}^{\alpha_1,\alpha_2}\), 
up to duality and twist, coincide with the bundles~\(\cE^{a,b}\)
that can be defined as truncations of the Koszul complex, see~\eqref{eq:ck-x-my},
or as simple extensions~\eqref{eq:ck-filtration}.
These bundles first appeared in~\cite[Conjecture~9.8]{KuzPol} and were extensively studied in~\cite{F15}.
The crucial computation with these objects is carried out in Proposition~\ref{prop:ck-complex},
where we assemble the restrictions to~\(\IGr(2,V)\) of the bundles~\(\cE^{a,b}\) into a complex,
and relate this complex to two Koszul complexes of~\eqref{eq:def-csk}.
This allows us to compute the cohomology of these complexes, see Theorem~\ref{thm:bicomplex},
and after dualization and twist, deduce Theorem~\ref{thm:main} in the case~\(k = 2\). 

To deduce Theorem~\ref{thm:main} for arbitrary~\(k\), 
we use the natural Fourier--Mukai functors
\begin{equation*}
\tilde\Phi_k \colon \Db(\Gr(2,V)) \to \Db(\Gr(k,V)), 
\qquad\text{and}\qquad
\Phi_k \colon \Db(\IGr(2,V)) \to \Db(\IGr(k,V)), 
\end{equation*}
see Section~\ref{ss:fm}.
In Lemmas~\ref{lem:tphi-cu} and~\ref{lem:phi-cs} we compute their actions
on appropriate Schur functors applied to the dual tautological bundle~\(\cU_2^\vee\) of~\(\IGr(2,V)\),
and on the symplectic wedge powers of the symplectic bundle~\(\cS_2\) of~\(\IGr(2,V)\).
We use the first to show that in some cases~\(\tilde\Phi_k\) 
takes a staircase complex on~\(\Gr(2,V)\) to a staircase complex on~\(\Gr(k,V)\),
see Proposition~\ref{prop:thpi-staircase},
and then deduce from this that 
it takes the objects~\(\cK_{\Gr(2,V)}^{\alpha_1,\alpha_2}\) to~\(\cK_{\Gr(k,V)}^{\alpha_1,\alpha_2}\),
see Corollary~\ref{cor:phi-ck}.
After that, Theorem~\ref{thm:main} easily follows.

\subsection*{Notation and conventions}

We work over an algebraically closed field~\(\Bbbk\) of characteristic zero.

\subsection*{Acknowledgments} I would like to thank my advisor, Alexander Kuznetsov, without whom this work would not have been possible, Lyalya Guseva for attention to my work and the anonymous referee for useful suggestions.

\section{Grassmannians of isotropic lines}
\label{sec:case_k=2}

Throughout this section we use the notation~\(\cU \coloneqq \cU_2\) 
for the tautological bundle on~\(\Gr(2,V)\) of rank~2 as well as for its restriction to~\(\IGr(2,V)\)
and~\(\cS \coloneqq \cS_2\) for the symplectic vector bundle (defined by~\eqref{eq:def-csk}) of rank~\(2n - 4\) on~\(\IGr(2,V)\).

The main result of the section is a proof of Theorem~\ref{thm:main} for~\(\IGr(2,V)\).

\subsection{Vector bundles~\texorpdfstring{\(\cE^{a,b}\)}{E{a,b}}}

Define vector bundles~\(\cE^{a,b}\) on~\(\Gr(2,V)\) for any
\begin{equation}
	\label{eq:ab-condition}
	0 \le a,b; \qquad a + b \le 2n - 2
\end{equation}
as a special case of the more general class of bundles~\(\cE^{\lambda,\mu}\) from~\cite{F15}:
\begin{equation}
	\cE^{a,b} \coloneqq \cE^{(b),(1^{2n-2-a})}(1), \qquad \text{where~\((1^{2n-2-a}) = (\underbrace{1, \dots, 1}_{2n-2-a})\).}
\end{equation}
In particular, the definition~\eqref{eq:K=E*} of~\(\cK_{\Gr(2,V)}\) implies
\begin{equation}
	\label{eq:E=K^*}
	\cE^{a,b} \cong \left(\cK_{\Gr(2,V)}^{2n-2-a,b}\right)^\vee(1).
\end{equation}
For our purposes it is, again, more convenient to interpret these bundles in terms of staircase complexes, 
rather than as pushforwards from a partial flag variety. Note that the combination of~\eqref{eq:E=K^*} with~\eqref{eq:def-ckk} already provides the left resolution:
\begin{multline}
	\label{eq:E-left_res}
	0 \to
	\mathrm{S}^{c-b}\cU(-1-b) \to 
	V \otimes \mathrm{S}^{c-b-1}\cU(-1-b) \to 
	\dots \to
	\wedge^{c-b}V \otimes \cO(-1-b)
	\\ \to 
	\wedge^{c-b+2}V^\vee \otimes \cO(-b) \to 
	\wedge^{c-b+3}V^\vee \otimes \cU(-b+1) \to 
	\dots \to
	\wedge^{c+1}V^\vee \otimes \mathrm{S}^{-b+1}\cU(-1) \\
	\to \cE^{a,b} \to
	0,
\end{multline}
where~\(c=2n-2-a\).

Next, we recall the definition of staircase complexes in the case of~\(\IGr(2, V)\).
For any pair of integers~\(\alpha_1 \ge \alpha_2\) we consider the \(\GL(V)\)-equivariant vector bundle
\begin{equation}
\label{eq:sigma-s}
\Sigma^{\alpha_1,\alpha_2}\cU^\vee \cong \Sy^{\alpha_1-\alpha_2}\cU^\vee \otimes \cO(\alpha_2) \cong \Sy^{\alpha_1-\alpha_2}\cU \otimes \cO(\alpha_1).
\end{equation}
By~\cite{Fonarev_2013} for any weight~\((\alpha_1,\alpha_2)\) of~\(\GL_2\) such that
\begin{equation}
\label{eq:weight-gl2}
\alpha_1 \ge \alpha_2 \ge \alpha_1 - 2n + 2
\end{equation}
there is an exact sequence of vector bundles, called the {\sf staircase complex}, that looks like
\begin{multline}
\label{eq:staircase-2lines}
0 \to
\wedge^{2n}V^\vee \otimes \Sigma^{\alpha_2-1,\alpha_1-2n+1}\cU^\vee \to 
\dots \to
\wedge^{\alpha_1-\alpha_2+2}V^\vee \otimes \Sigma^{\alpha_2-1,\alpha_2-1}\cU^\vee 
\\ \to 
\wedge^{\alpha_1-\alpha_2}V^\vee \otimes \Sigma^{\alpha_2,\alpha_2}\cU^\vee \to 
\dots \to
V^\vee \otimes \Sigma^{\alpha_1-1,\alpha_2}\cU^\vee \to 
\Sigma^{\alpha_1,\alpha_2}\cU^\vee \to 
0.
\end{multline}
Here in the top row the second component of the weight increases from~\(\alpha_1 - 2n + 1\) to~\(\alpha_2-1\), 
while the exponent of the wedge power decreases from~\(2n\) to~\(\alpha_1 - \alpha_2 + 2\),
and in the bottom row the first component of the weight increases from~\(\alpha_2\) to~\(\alpha_1\), 
while the exponent of the wedge power decreases from~\(\alpha_1 - \alpha_2\) to~\(0\).

\begin{remark}
\label{rem:staircase-unique}
The staircase complex is \(\GL(V)\)-equivariant,
and its differentials are the \emph{unique non-zero \(\GL(V)\)-equivariant maps} between the respective terms.
In particular, the dual of the staircase complex~\eqref{eq:staircase-2lines} with parameters~\((\alpha_1,\alpha_2)\)
is the staircase complex with parameters~\((2n - 1 - \alpha_1, 1 - \alpha_2)\).
It also follows that the bottom line of~\eqref{eq:staircase-2lines} is, up to twist, the Koszul complex
\begin{equation}
\label{eq:koszul-right}
\Big\{ 
\wedge^mV^\vee \otimes \cO \to 
\wedge^{m-1}V^\vee \otimes \cU^\vee \to 
\dots \to 
V^\vee \otimes \Sy^{m-1}\cU^\vee \to 
\Sy^m\cU^\vee
\Big\},
\end{equation}
where~\(m = \alpha_1 - \alpha_2\),
and the top line of~\eqref{eq:staircase-2lines} is, up to twist, the Koszul complex
\begin{equation*}
\Big\{ 
\Sy^m\cU \to V \otimes \Sy^{m-1}\cU \to 
\dots \to 
\wedge^{m-1}V \otimes \cU \to 
\wedge^mV \otimes \cO
\Big\},
\end{equation*}
where~\(m = 2n - 2 - \alpha_1 + \alpha_2\).
\end{remark}

We consider the special case of staircase complexes where~\(\alpha_1 \ge 0 \ge \alpha_2\), 
and for convenience we write~\((\alpha_1,\alpha_2) = (a,-b)\).
Then the condition~\eqref{eq:weight-gl2} translates into the inequalities~\eqref{eq:ab-condition}.
Using the isomorphism~\(\wedge^iV \cong \wedge^{2n-i}V^\vee\) 
and the uniqueness of staircase complexes explained in Remark~\ref{rem:staircase-unique},
note that the resolution~\eqref{eq:E-left_res} is a truncation of the exact sequence~\eqref{eq:staircase-2lines}. 
So, we finally obtain our right resolution of the bundle~\(\cE^{a,b}\):
\begin{equation}
\label{eq:ck-x-my}
0 \to \cE^{a,b} \to 
\wedge^{a}V^\vee \otimes \Sigma^{0,-b}\cU^\vee \to 
\dots \to
V^\vee \otimes \Sigma^{a-1,-b}\cU^\vee \to 
\Sigma^{a,-b}\cU^\vee \to 
0.
\end{equation} 
It is the truncation of the staircase complex~\eqref{eq:staircase-2lines} with~\((\alpha_1,\alpha_2) = (a,-b)\)
at the term with weight~\((0,-b)\). 
Alternatively, \(\cE^{a,b}\) is a truncation of the Koszul complex~\eqref{eq:koszul-right} with~\(m = a + b\) twisted by~\(\cO(-b)\).
Note that~\(\Sigma^{0,-b}\cU^\vee \cong \Sy^b\cU\), so~\(\cE^{a,b}\) 
is a \(\GL(V)\)-equivariant subbundle in~\(\wedge^{a} V^\vee \otimes \Sy^b\cU\). In particular,~\(\cE^{0,b} \cong \Sy^b\cU\).

Next, if~\(a \ge 1\) by~\cite[Theorem~4.4]{F15} there is a~\(\GL(V)\)-equivariant exact sequence:
\begin{equation}
	\label{eq:ce-ce-es}
	0 \to \cE^{a,b} \to \wedge^{a} V^\vee \otimes \Sy^b\cU \to \cE^{a-1,b+1}(1) \to 0.
\end{equation}

Further, the tautological sequence~\eqref{eq:def-uperp}
induces a filtration on the bundle \(\wedge^{a} V^\vee \otimes \Sy^b\cU\) with factors:
\begin{equation}\label{eq:V-filtration}
	\wedge^{a} \cU^\perp \otimes \Sy^b\cU,
	\quad
	\wedge^{a-1} \cU^\perp \otimes \Sy^{b-1}\cU, 
	\quad
	\wedge^{a-1} \cU^\perp \otimes \Sy^{b+1}\cU(1),
	\quad
	\wedge^{a-2} \cU^\perp \otimes \Sy^b\cU(1).
\end{equation}
Where, by convention, all wedge and symmetric powers~\(\wedge^i(-)\) and~\(\Sy^i(-)\) with negative~\(i\) are zero. 

By induction on~\(a \ge 0\) one gets the following exact sequence for~\(\cE^{a,b}\), since we already know it for~\(a = 0\):
\begin{equation}
	\label{eq:ck-filtration}
	0 \to \wedge^{a} \cU^\perp \otimes \Sy^b\cU \to \cE^{a,b} \to \wedge^{a-1} \cU^\perp \otimes \Sy^{b-1}\cU \to 0.
\end{equation} 

\begin{remark}
	By~\cite[Proposition~3.5]{F15} the bundles~\(\cE^{a,b}\) are exceptional.
	In particular, the extension~\eqref{eq:ck-filtration} is non-split.
\end{remark}

\subsection{Secondary staircase complexes on~\texorpdfstring{\(\IGr(2,V)\)}{IGr(2,V)}}

Now, we consider the restrictions of the bundles~\(\cE^{a,b}\) to~\(\IGr(2,V)\).
Abusing notation, we will still write~\(\cE^{a,b}\) for these restrictions.
In this subsection we construct \(\Sp(V)\)-equivariant complexes on~\(\IGr(2,V)\) from these bundles. 
Recall that~\(\omega\) denotes the symplectic form on~\(V\).

First, we consider two \(\Sp(V)\)-equivariant maps between the ambient bundles of~\(\cE^{a,b}\) and~\(\cE^{a+1,b-1}\):
\begin{equation*}
\vn_1,\vn_2 \colon \wedge^a V^\vee \otimes \Sy^b\cU \to \wedge^{a+1}V^\vee \otimes \Sy^{b-1}\cU.
\end{equation*}
Denote the composition of~\(V^\vee \twoheadrightarrow \cU^\vee\) and the trace map~\(\cU^\vee \otimes \cU \to \cO\) by
\begin{equation*}
	\operatorname{tr} \colon V^\vee \otimes \cU \to \cO.
\end{equation*}
It induces the map 
\begin{equation*}
	\operatorname{tr} \colon \wedge^a V^\vee \otimes \Sy^b\cU \to \wedge^{a-1}V^\vee \otimes \Sy^{b-1}\cU.
\end{equation*}
The first morphism is defined as the composition of~\(\operatorname{tr}\) and~\(- \wedge \omega 
\):
\begin{equation*}
	\vn_1 \colon \wedge^a V^\vee \otimes \Sy^b\cU \xrightarrow{\operatorname{tr}} \wedge^{a-1}V^\vee \otimes \Sy^{b-1}\cU \xrightarrow{- \wedge \omega} \wedge^{a+1}V^\vee \otimes \Sy^{b-1}\cU.
\end{equation*}
The second one is also induced by the composition of multiplication by~\(\omega\) and~\(\operatorname{tr}\), but in the opposite order:
\begin{equation*}
	\vn_2 \colon \wedge^a V^\vee \otimes \Sy^b \cU \xrightarrow{- \otimes \omega \otimes -} \wedge^a V^\vee \otimes \left(\wedge^2 V^\vee \otimes \Sy^{b}\cU\right)
	\xrightarrow{- \otimes \operatorname{tr}}  \wedge^{a} V^\vee \otimes V^\vee \otimes \Sy^{b-1}\cU  \to \wedge^{a+1}V^\vee \otimes \Sy^{b-1}\cU,
\end{equation*}
where the last arrow is just a canonical surjection on the direct summand.

\begin{remark}
	The abstract descriptions of~\(\vn_1\) and~\(\vn_2\) are not easy to work with. Specifically, it is difficult to compute their components between the subquotients of the filtrations~\eqref{eq:V-filtration} and~\eqref{eq:ck-filtration}, which is our main goal. 
	
	On the other hand, the direct computation in the case~\(k=2\) is straightforward, unambiguous and turns out to be quite simple. 
\end{remark}

Fix a point~\([U] \in \IGr(2,V)\) and choose a basis~\(e_1,e_2\) of~\(U\). 
Then~\(\operatorname{tr}\) acts at this point by a convolution with~\(\sum e_i \otimes \partial/\partial e_i\).
So, the maps are given by the formulas:
\begin{align}
\label{eq:def-v1}
\vn_1(\lambda \otimes P) & \coloneqq 
\lambda_1 \wedge \omega \otimes P_1 + 
\lambda_2 \wedge \omega \otimes P_2,
\\
\label{eq:def-v2}
\vn_2(\lambda \otimes P) & \coloneqq 
\lambda \wedge (\omega_1 \otimes P_1 + \omega_2 \otimes P_2).
\end{align}
where~\(\lambda \in \wedge^aV^\vee\) and~\(P \in \Sy^b U\) is considered as a homogeneous polynomial of degree~\(b\) on~\(U^\vee\).
Moreover, we write~\(\lambda_i\) and~\(\omega_i\) for the evaluation of the corresponding skew-form on~\(e_i\)
and~\(P_i\) for the derivative of~\(P\) in the direction~\(e_i\).

The maps~\(\vn_1\) and~\(\vn_2\) are obviously independent of the choice of basis.

\begin{proposition}
\label{prop:ck-complex}
For any~\((a,b)\) such that~\eqref{eq:ab-condition} holds,
the map
\begin{equation*}
\vn \coloneqq \tfrac1{b + 1}\vn_1 + \vn_2 \colon \wedge^aV^\vee \otimes \Sy^b\cU \to \wedge^{a+1}V^\vee \otimes \Sy^{b-1}\cU
\end{equation*}
takes the subbundle~\(\cE^{a,b} \subset \wedge^aV^\vee \otimes \Sy^b\cU\) to~\(\cE^{a+1,b-1} \subset \wedge^{a+1}V^\vee \otimes \Sy^{b-1}\cU\).
Moreover, the chain 
\begin{equation}
\label{eq:ck-complex}
\cE^{a,b} \xrightarrow{\ \vn\ }
\cE^{a+1,b-1} \xrightarrow{\ \vn\ }
\dots \xrightarrow{\ \vn\ }
\cE^{a+b-1,1} \xrightarrow{\ \vn\ }
\cE^{a+b,0}
\end{equation}
is a complex of vector bundles on~\(\IGr(2,V)\).
Finally, the maps~\(\vn\) preserve the filtrations~\eqref{eq:ck-filtration}
and give rise to 
an exact sequence of complexes
\[\begin{tikzcd}[column sep=small]
	0 \ar[d] & 0 \ar[d] && 0 \ar[d] & 0 \ar[d] &
	\\
	\wedge^a \cU^\perp \otimes \Sy^b \cU \ar[r, "{\vn_2}"] \ar[d] &
	\wedge^{a+1} \cU^\perp \otimes \Sy^{b-1} \cU \ar[r,"{\vn_2}"] \ar[d] &
	\dots \ar[r,"{\vn_2}"] &
	\wedge^{a+b-1} \cU^\perp \otimes \cU \ar[r, "{\vn_2}"] \ar[d] &
	\wedge^{a+b} \cU^\perp \ar[d] 
	\\
	\cE^{a,b} \ar[r, "\vn"] \ar[d] &
	\cE^{a+1,b-1} \ar[r,"\vn"] \ar[d] &
	\dots \ar[r,"\vn"] &
	\cE^{a+b-1,1} \ar[r,"\vn"] \ar[d] &
	\cE^{a+b,0} \ar[d] 
	\\
	\wedge^{a-1} \cU^\perp \otimes \Sy^{b-1} \cU \ar[r,"{-\vn_2}"] \ar[d] &
	\wedge^{a} \cU^\perp \otimes \Sy^{b-2} \cU \ar[r, "{-\vn_2}"] \ar[d] &
	\dots \ar[r,"{-\vn_2}"] &
	\wedge^{a+b-2} \cU^\perp \ar[r] \ar[d] &
	0
	\\
	0 & 0 && 0
\end{tikzcd}\]
whose top and bottom rows are the truncated Koszul complexes of the exact sequence~\eqref{eq:def-csk}.
\end{proposition}

\begin{proof}
We verify all the clams at one point~\([U] \in \IGr(2,V)\).
Since all vector bundles and morphisms involved are~\(\Sp(V)\)-equivariant, this is enough to prove the proposition.

We start by computing the map~\(\vn\) on the first factor~\(\wedge^a U^\perp \otimes \Sy^b U\) of the filtration~\eqref{eq:ck-filtration}.
Then~\mbox{\(\lambda \in \wedge^aU^\perp\)}, hence~\(\lambda_1 = \lambda_2 = 0\), and therefore
the map~\(\vn_1\) vanishes on this factor.
On the other hand, since~\(\omega_1,\omega_2 \in U^\perp\) (because~\(U\) is isotropic), 
we have~\(\lambda \wedge \omega_1, \lambda \wedge \omega_2 \in \wedge^{a+1}U^\perp\), 
hence~\(\vn_2\) preserves the first factor. 
It is also clear, that~\(\vn_2\) coincides on the first factor with the differential of the Koszul complex
\begin{equation}
\label{eq:koszul-cs}
0 \to
\Sy^t \cU \xrightarrow{\ \vn_2\ }
\cU^\perp \otimes \Sy^{t-1} \cU \xrightarrow{\ \vn_2\ }
\dots \xrightarrow{\ \vn_2\ }
\wedge^{t -1} \cU^\perp \otimes \cU \xrightarrow{\ \vn_2\ }
\wedge^{t} \cU^\perp \to
\wedge^t\cS \to 
0,
\end{equation}
of the short exact sequence~\eqref{eq:def-csk}.

Now, consider the second factor of the filtration~\eqref{eq:ck-filtration}.
To describe it explicitly, we consider the dual basis of~\(e_1,e_2 \in U\)
and lift its elements to linear functions~\(e^1,e^2 \in V^\vee\).
This allows us to lift an element~\(\mu \otimes Q\) 
from the second factor~\(\wedge^{a-1}U^\perp \otimes \Sy^{b-1}U\) of~\eqref{eq:ck-filtration} 
to~\(\wedge^aV^\vee \otimes \Sy^bU\) as
\begin{equation*}
\xi \coloneqq (\mu \wedge e^1) \otimes (e_1 \cdot Q) + (\mu \wedge e^2) \otimes (e_2 \cdot Q),
\end{equation*}
where~\(\mu \in \wedge^{a-1}U^\perp\) and~\(Q \in \Sy^{b - 1}U\) is a polynomial of degree~\(b - 1\) in~\(e_1\) and~\(e_2\).
Then
\begin{equation*}
\vn_1(\xi) = 
\mu \wedge (\omega \otimes 2Q + \omega \otimes e_1 \cdot Q_1 + \omega \otimes e_2 \cdot Q_2) =
(b + 1)\mu \wedge \omega \otimes Q,
\end{equation*}
where~\(Q_1\) and~\(Q_2\) are the derivatives of~\(Q\) with respect to~\(e_1\) and~\(e_2\).
Similarly,
\begin{align*}
\vn_2(\xi) 
&= \mu \wedge (e^1 \wedge \omega_1 \otimes (Q + e_1 \cdot Q_1) + e^1 \wedge \omega_2 \otimes e_1 \cdot Q_2) \\ 
&+ \mu \wedge (e^2 \wedge \omega_1 \otimes e_2 \cdot Q_1 + e^2 \wedge \omega_2 \otimes (Q + e_2 \cdot Q_2).
\end{align*}
Now consider the skew-form
\begin{equation}
\label{eq:def-bar-omega}
\bar\omega \coloneqq \omega + e^1 \wedge \omega_1 + e^2 \wedge \omega_2 \in \wedge^2V^\vee.
\end{equation} 
Of course, it depends on the choice of lifts~\(e^i\), 
but for any choice of the lifts the form~\(\bar\omega\) is contained in~\mbox{\(\wedge^2U^\perp \subset \wedge^2V^\vee\)}.
Using this notation and the above computations, we can write
\begin{equation*}
\vn(\xi) = \mu \wedge \bar\omega \otimes Q 
- \mu \wedge \Big(\omega_1 \wedge (e^1 \otimes e_1 + e^2 \otimes e_2) \cdot Q_1 
+ \omega_2 \wedge (e^1 \otimes e_1 + e^2 \otimes e_2) \cdot Q_2\Big),
\end{equation*}
and since~\(\bar\omega \in \wedge^2U^\perp\), 
we see that the first summand is contained in the first factor of the filtration~\eqref{eq:ck-filtration},
while the second summand is in its second factor.
It follows that 
\begin{equation*}
\vn(\cE^{a,b}) \subset \cE^{a+1,-b+1},
\end{equation*}
and thus the first statement of the proposition is proved.

Moreover, it follows that~\(\vn\)
acts on~\(\cE^{a,b}\vert_{[U]} = (\wedge^a U^\perp \otimes \Sy^b U) \oplus (\wedge^{a-1} U^\perp \otimes \Sy^{b-1} U)\) 
(where the direct sum decomposition is induced by the lifts~\(e^i\) chosen above)
by the matrix
\begin{equation*}
\begin{pmatrix}
\vn_2 &
\bar\omega \\
0 &
-\vn_2.
\end{pmatrix}
\end{equation*}
Thus, to deduce the second and third statements, 
it remains to note that the wedge product with~\(\bar\omega\) commutes with the Koszul differential~\(\vn_2\),
which is obvious from~\eqref{eq:def-v2}.
\end{proof}

\begin{remark}
	\label{rem:E^bullet_unique}
	Using~\cite[Corollary~3.10]{F15} it is easy to check that there are unique maps of vector bundles from~\(\cE^{a,b}\) to \(\cE^{a+1,b-1}\) and no higher morphisms. Analogously,~\(\cE^{a,b}\) is semiorthogonal to~\(\cE^{a+i,b-i}\) for \(i \ge 2\). This allows us to construct the complex~\eqref{eq:ck-complex} abstractly, as well as the complex~\eqref{eq:koszul-cs}. However, it is not clear whether they form the short exact sequence of complexes from the Proposition~\ref{prop:ck-complex} via~\eqref{eq:ck-filtration}. 
\end{remark}

\subsection{Cohomology of secondary staircase complexes on~\texorpdfstring{\(\IGr(2,V)\)}{IGr(2,V)}}

Now we consider the secondary staircase complex~\eqref{eq:ck-complex} for~\(a=0\). 
Recall that~\(\cS \coloneqq \cU^\perp/\cU\) is the symplectic vector bundle of rank~\(2n - 4\) on~\(\IGr(2,V)\), see~\eqref{eq:def-csk},
\(\omega_S \in H^0(\IGr(2,V), \wedge^2\cS)\) is its symplectic form,
and~\(\wedge_\Sp^i\cS\), \(0 \le i \le n - 2\), denote its symplectic wedge powers, see~\eqref{eq:def-wedge-sp}.

\begin{theorem}
\label{thm:bicomplex}
For any~\(0 \le t \le 2n - 2\) consider the complex
\begin{equation}
\label{eq:ck-special}
\cE_t^\bullet \coloneqq
\Big\{\cE^{0,t} \xrightarrow{\ \vn\ }
\cE^{1,t-1} \xrightarrow{\ \vn\ }
\dots \xrightarrow{\ \vn\ }
\cE^{t-1,-1} \xrightarrow{\ \vn\ }
\cE^{t,0}\Big\},
\end{equation}
where the rightmost term is in degree~\(0\).
Then its cohomology sheaves are the following:
\begin{equation}
\label{eq:ch-ck}
\cH^i(\cE_t^\bullet) \cong
\begin{cases}
\wedge^t_\Sp\cS, & \text{if~\(i = \hphantom{-}0\) and~\(0 \le t \le n - 2\),}\\
\wedge^{2n-2-t}_\Sp\cS, & \text{if~\(i = -1\) and~\(n \le t \le 2n - 2\),}\\
0, & \text{otherwise}.
\end{cases}
\end{equation}
\end{theorem}

\begin{proof}
We consider the complex of Proposition~\ref{prop:ck-complex} for~\(a = 0\), \(b = t\).
Applying the snake lemma and taking~\eqref{eq:koszul-cs} into account, 
we obtain the exact sequence
\begin{equation}
\label{eq:ch-ck-sequence}
0 \to \cH^{-1}(\cE_t^\bullet) \xrightarrow{\ \quad} \wedge^{t-2}\cS \xrightarrow{\quad\ } \wedge^t\cS \xrightarrow{\ \quad} \cH^0(\cE_t^\bullet) \to 0
\end{equation} 
and~\(\cH^i(\cE_t^\bullet) = 0\) for~\(i \not\in \{-1,0\}\).
Furthermore, the argument of Proposition~\ref{prop:ck-complex} shows that 
the middle map~\(\wedge^{t-2}\cS \to \wedge^t\cS\) in~\eqref{eq:ch-ck-sequence} at the point~\([U] \in \IGr(2,V)\)
is induced by the map 
\begin{equation*}
\bar\omega \colon \wedge^{t-2}U^\perp \to \wedge^tU^\perp,
\qquad 
\lambda \mapsto \lambda \wedge \bar\omega,
\end{equation*}
where~\(\bar\omega\) is defined in~\eqref{eq:def-bar-omega},
and therefore it coincides with the map induced by the symplectic form~\(\omega_\cS\) of~\(\cS\).
In particular, when~\(0 \le t \le n - 2\), the middle map in~\eqref{eq:ch-ck-sequence} is injective,
so its kernel is zero, and by~\eqref{eq:def-wedge-sp} its cokernel is~\(\wedge^t_\Sp\cS\).
Similarly, when~\(n \le t \le 2n - 2\) the middle map in~\eqref{eq:ch-ck-sequence} is surjective,
so its cokernel is zero, and by duality and~\eqref{eq:def-wedge-sp} its kernel is~\(\wedge^{2n-2-t}_\Sp\cS\).
Finally, if~\(t = n - 1\) this map is an isomorphism.
\end{proof}

\begin{remark}
\label{rem:bicomplex}
The case where~\(t = n - 1\) is particularly interesting.
In this case the complex~\(\cE_{n-1}^\bullet\) is acyclic.
On the other hand, its terms~\(\cE^{n-1-b,b}\) have resolutions~\eqref{eq:ck-x-my}.
Using~\cite[Lemma~5.1]{38} it is easy to check that the morphisms of~\(\cE_{n-1}^\bullet\)
extend in a unique way to morphisms between these resolutions,
so that we obtain a bicomplex whose terms are~\(\wedge^{n-1-b-c}V^\vee \otimes \Sigma^{c,-b}\cU^\vee\), see~\eqref{eq:bicomplex}.
In fact, these unique extensions can be written very explicitly, we do this in Corollary~\ref{cor:bicomplex-maps}.

Note that the bicomplex~\eqref{eq:bicomplex} has exactly the same form as the bicomplex of~\cite[Proposition~5.3]{lines} 
that was crucial for the proof of fullness of the exceptional collection on~\(\IGr(2,2n)\).
Thus, our computation shows that the formula for the maps in the bicomplex suggested in~\cite{lines} is incorrect, see Remark~\ref{rem:bicomplex2}. Therefore Theorem~\ref{thm:bicomplex} fills the gap in the proof of~\cite[Proposition~5.3]{lines}.
\end{remark}

Now we can prove the first case of our main theorem.

\begin{proof}[of Theorem~\textup{\ref{thm:main}} for~\(k = 2\)]
Recall that the bundle~\(\cE^{a,b}\) satisfies~\eqref{eq:E=K^*}.
Therefore, dualizing~\eqref{eq:ck-special} and twisting it by~\(\cO(-1)\), we obtain the complex~\eqref{eq:ckk-special}.
Now the isomorphisms~\eqref{eq:ch-cckk} follow from~\eqref{eq:ch-ck}, because the symplectic bundle~\(\cS = \cS_2\) is self-dual. 
\end{proof}

\section{General isotropic Grassmannians}

In this section we construct secondary staircase complexes and compute their cohomology
on the isotropic Grassmanians~\(\IGr(k,V)\) with~\(3 \le k \le n\).
The main tool is the natural Fourier--Mukai functor 
between the derived categories of Grassmannians and isotropic Grassmannians.

\subsection{Fourier--Mukai transforms}
\label{ss:fm}

Consider the commutative diagram of (isotropic) partial flag varieties
and their natural embeddings and projections to (isotropic) Grassmannians:
\begin{equation}
\label{eq:diagram-igr-gr}
\vcenter{
\xymatrix@C=3em{
\IGr(2,V) \ar[d]_{\iota_2} &
\IFl(2,k;V) \ar[d]_{\iota} \ar[l]_{\pi_2} \ar[r]^{\pi_k} &
\IGr(k,V) \ar[d]^{\iota_k} 
\\
\Gr(2,V) &
\Fl(2,k;V) \ar[l]_{\tilde\pi_2} \ar[r]^{\tilde\pi_k} &
\Gr(k,V).
}}
\end{equation}
Consider also the induced Fourier--Mukai functors
\begin{align*}
\Phi_k & \coloneqq (\pi_k)_* \circ \pi_2^* \colon \Db(\IGr(2,V)) \to \Db(\IGr(k,V)),\\
\tilde\Phi_k & \coloneqq (\tilde\pi_k)_* \circ \tilde\pi_2^* \colon \Db(\Gr(2,V)) \to \Db(\Gr(k,V)),
\end{align*}
where the pullbacks and pushforwards are derived.
We start with a simple observation.

\begin{lemma}
\label{lem:phi-tphi}
For any object~\mbox{\(\cF \in \Db(\Gr(2,V))\)} we have~\(\Phi_k(\iota_2^*\cF) \cong \iota_k^*\tilde\Phi_k(\cF)\).
\end{lemma}

\begin{proof}
Indeed, the right square in~\eqref{eq:diagram-igr-gr} is Cartesian with flat horizontal arrows, hence
\begin{equation*}
\Phi_k(\iota_2^*\cF) =
(\pi_k)_*\pi_2^*\iota_2^*\cF \cong
(\pi_k)_*\iota^*\tilde\pi_2^*\cF \cong
\iota_k^*(\tilde\pi_k)_*\tilde\pi_2^*\cF =
\iota_k^*\tilde\Phi_k(\cF),
\end{equation*}
where the first isomorphism follows from commutativity of the left square,
and the second is the base change formula for the right square.
\end{proof}

We apply these functors to the simplest irreducible equivariant vector bundles.
In the computations, we use Borel--Bott--Weil theorem for the morphisms~\(\tilde\pi_k\) and~\(\pi_k\).
Note that these morphisms are fibrations with fiber~\(\Gr(2,k)\),
so to compute the derived pushforward of an equivariant vector bundle 
associated with a weight~\(\beta\) of~\(\GL(V)\) or~\(\Sp(V)\),
we consider the sum~\(\bar\beta\) of the first \(k\)-components of the weight~\(\beta\)
with the special weight
\begin{equation*}
\rho_k \coloneqq (k,k-1,\dots,1)
\end{equation*}
of the group~\(\GL_k\).
If all the components of~\(\bar\beta\) are distinct, 
the derived pushforward is isomorphic to the equivariant bundle whose first~\(k\) components 
are given by the weight~\(\sigma(\bar\beta) - \rho_k\)
(here~\(\sigma \in \mathfrak{S}_k\) is the minimal permutation such that the components of~\(\sigma(\bar\beta)\) are decreasing)
and the remaining components coincide with those of the weight~\(\beta\); 
this equivariant bundle is shifted in the derived category to the right by the length of~\(\sigma\).
Otherwise (if some components of~\(\bar\beta\) coincide), the derived pushforward is zero.

Now we do the computations: the first is very easy.

\begin{lemma}
\label{lem:tphi-cu}
Assume~\(\alpha_1 \ge \alpha_2\) and~\(\alpha_1 \ge -1\).
Then
\begin{equation}
\label{eq:thpik-sigma}
\tilde\Phi_k(\Sigma^{\alpha_1,\alpha_2}\cU_2^\vee) \cong 
\begin{cases}
\Sigma^{\alpha_1,\alpha_2,0,\dots,0}\cU_k^\vee, & \text{if~\(\alpha_2 \ge 0\)},\\
0, & \text{if~\(-1 \ge \alpha_2 \ge 2 - k\)},\\
\Sigma^{\alpha_1,-1,\dots,-1,k-2+\alpha_2}\cU_k^\vee[2-k], & \text{if~\(1-k \ge \alpha_2\)}.
\end{cases}
\end{equation}
\end{lemma}

\begin{proof}
The bundle~\(\Sigma^{\alpha_1,\alpha_2}\cU_2^\vee\) is \(\GL(V)\)-equivariant and the corresponding weight is
\begin{equation*}
\bar\beta = (\alpha_1,\alpha_2,0,\dots,0) + \rho_k = (k+\alpha_1,k-1+\alpha_2,k-2,\dots,1).
\end{equation*}
If~\(\alpha_2 \ge 0\) this weight is strictly dominant, so the first line of~\eqref{eq:thpik-sigma} follows.
If~\(-1 \ge \alpha_2 \ge 2 - k\), the second component of the weight is equal to one of the other components, 
hence the pushforward is zero, so the second line of~\eqref{eq:thpik-sigma} follows.
Finally, if~\(1 - k \ge \alpha_2\), to make the weight dominant, we need to move its second component to the far right.
The length of the corresponding permutation is~\(k - 2\), and since
\begin{equation*}
(k+\alpha_1,k-2,\dots,1,k-1+\alpha_2) - \rho_k = (\alpha_1,-1,\dots,-1,k-2+\alpha_2),
\end{equation*}
the third line of~\eqref{eq:thpik-sigma} also follows.
\end{proof}

The second computation is more complicated.

\begin{lemma}
\label{lem:phi-cs}
For~\(0 \le t \le n - 2\) we have
\begin{equation}
\label{eq:phi-k-wt-sp-cs-m1}
\Phi_k(\wedge^t_\Sp\cS_2(-1)) \cong 
\begin{cases}
0, & \text{if~\(0 \le t \le k - 3\)},\\
\wedge^{t-k+2}_\Sp\cS_k(-1), & \text{if~\(k - 2 \le t \le n - 2\)}.
\end{cases}
\end{equation}
\end{lemma}
\begin{proof}
The flag~\(0 \subset \cU_2 \subset \cU_k \subset \cU_k^\perp \subset \cU_2^\perp\) on~\(\IFl(2,k;V)\)
induces a filtration on~\(\pi_2^*\cS_2\) with factors
\begin{equation*}
\cU_k / \cU_2,
\qquad 
\cU_k^\perp/\cU_k \cong \pi_k^*(\cS_k),
\qquad 
\cU_2^\perp / \cU_k^\perp \cong (\cU_k / \cU_2)^\vee.
\end{equation*}
By~\eqref{eq:def-wedge-sp} we need to compute~\(\Phi_k(\wedge^m\cS_2(-1))\) for~\(m = t - 2\) and~\(m = t\) and take the cone of the morphism between them induced by~\(\omega_S\). 
Note that~\(\wedge^m\cS_2\) has a filtration with terms
\begin{equation*}
\wedge^i(\cU_k/\cU_2) \otimes \wedge^j(\cU_k/\cU_2)^\vee \otimes \pi_k^*(\wedge^{m-i-j} \cS_k),
\qquad 
0 \le i,j \le k - 2,\quad i + j \le m.
\end{equation*}

Using the Pieri formula we rewrite this product as
\begin{equation*}
\bigoplus_{s = \left\lceil\! \frac{i+j+2-k}{2} \!\right\rceil}^{\min(i,j)} 
\Sigma^{(1^{j-s},\,0^{k-2-i-j+2s},\,(-1)^{i-s})} (\cU_k/\cU_2)^\vee \otimes \pi_k^*(\wedge^{m-i-j} \cS_k).
\end{equation*}
To compute~\(\Phi_k(\wedge^m \cS_2(-1))\) we twist this filtration by~\(\pi_2^*(\cO(-1))\) and push forward along~\(\pi_k\).
Using the projection formula, we reduce the computation to the objects
\begin{equation*}
(\pi_k)_*\Big(\Sigma^{(1^{j-s},\,0^{k-2-i-j+2s},\,(-1)^{i-s})} (\cU_k/\cU_2)^\vee \otimes \pi_2^*(\cO(-1))\Big).
\end{equation*}
The corresponding weight is~\((-1,-1,1^{j-s},\,0^{k-2-i-j+2s},\,(-1)^{i-s})\)
and its sum with~\(\rho_k\) is 
\begin{equation}
\label{eq:weight}
\bar\beta = 
(k-1,k-2,
\underbrace{k-1,\dots,k-j+s}_{\text{\(j-s\) terms}}, 
\underbrace{k-j+s-2,\dots,i-s+1}_{\text{\(k-2-i-j+2s\) terms}},
\underbrace{i-s-1,\dots,0}_{\text{\(i-s\) terms}}).
\end{equation}
If~\(j - s > 0\), the first and third coordinates are the same, hence the pushforward is zero.
If~\(j - s = 0\) but~\(k-2-i-j+2s = k - 2 - i + s > 0\), the second and third coordinates are the same, hence the pushforward is zero again.
Thus, the only case where the pushforward is nonzero is the case~\(s = j = i + 2 - k\).
But since~\(i \le k - 2\) and~\(j \ge 0\), we conclude that we must have~\(s = j = 0\) and~\(i = k - 2\).
In this case the weight~\eqref{eq:weight} equals~\((k -1, k - 2, k - 3, \dots, 0)\),
and the pushforward of the corresponding term of the filtration of~\(\pi_2^*(\wedge^m \cS_2(-1))\) is 
\begin{equation*}
(\pi_k)_*\Big(\Sigma^{(-1,-1,\dots,-1)} (\cU_k/\cU_2)^\vee \otimes \pi_2^*(\cO(-1)) \otimes \pi_k^*(\wedge^{m - k + 2} \cS_k) \Big) 
\cong \wedge^{m-k+2} \cS_k(-1).
\end{equation*}
In particular, the irreducible bundles have no higher pushforwards.
Thus, 
we get:
\begin{equation*}
	\Phi_k(\wedge^t_\Sp \cS_2(-1)) = \operatorname{Coker}(\wedge^{t-k} \cS_k(-1) \overset{\omega_\cS}\hookrightarrow \wedge^{t-k+2} \cS_k(-1)).
\end{equation*}
Finally, note that both terms vanish
if~\(t \le k - 3\), 
which gives us the first line of~\eqref{eq:phi-k-wt-sp-cs-m1},
and if~\(k - 2 \le t \le n - 2\), we obtain the second line.
\end{proof}

\subsection{General secondary staircase complexes}

In this subsection we show that the functor~\(\tilde\Phi_k\) 
takes some staircase complexes of~\(\Gr(2,V)\) to staircase complexes of~\(\Gr(k,V)\).
As a consequence, we show that~\(\tilde\Phi_k\) takes bundles~\(\cK_{\Gr(2,V)}^{\alpha_1,\alpha_2}\) to bundles~\(\cK_{\Gr(k,V)}^{\alpha_1,\alpha_2}\).

The general definition of a staircase complex on~\(\Gr(k,V) = \Gr(k,2n)\) can be found in~\cite{Fonarev_2013}.
Such a complex is defined for any weight~\(\alpha\) of~\(\GL_k\) satisfying the condition
\begin{equation*}
\alpha_1 \ge \alpha_2 \ge \alpha_3 \ge \dots \ge \alpha_k \ge \alpha_1 - 2n + k .
\end{equation*}
For our purposes it is enough to consider the case where~\(\alpha_3 = \dots = \alpha_k = 0\), hence
\begin{equation*}
2n - k \ge \alpha_1 \ge \alpha_2 \ge 0.
\end{equation*}
Under this assumption the staircase complex takes the form
\begin{multline}
\label{eq:staircase-3-lines}
0 \to 
\wedge^{2n}V^\vee \otimes \Sigma^{\alpha_2-1,-1,\dots,-1,\alpha_1-2n+k-1}\cU^\vee \to 
\dots \to 
\wedge^{\alpha_1+k+2} V^\vee \otimes \Sigma^{\alpha_2-1,-1,-1,\dots,-1}\cU^\vee \\ \qquad\quad  \to 
\wedge^{\alpha_1+1} V^\vee \otimes \Sigma^{\alpha_2-1,0,0,\dots,0}\cU^\vee \to 
\dots \to 
\wedge^{\alpha_1-\alpha_2+2} V^\vee \otimes \Sigma^{\alpha_2-1,\alpha_2-1,0,\dots,0}\cU^\vee \\ \to 
\wedge^{\alpha_1-\alpha_2} V^\vee \otimes \Sigma^{\alpha_2,\alpha_2,0,\dots,0}\cU^\vee \to 
\dots \to
\Sigma^{\alpha_1,\alpha_2,0,\dots,0}\cU^\vee \to 
0.
\end{multline}
Here in the top row the last component  of the weight gradually increases,
in the middle row the second component does the same,
and in the bottom row the first component does so.
In the special case where~\(\alpha_2 = 0\) the middle row disappears.

\begin{proposition}
\label{prop:thpi-staircase}
If~\(2n - k \ge \alpha_1 \ge \alpha_2 \ge 0\), the staircase complex~\eqref{eq:staircase-3-lines} 
is the image under the functor~\(\tilde\Phi_k\)
of the staircase complex~\eqref{eq:staircase-2lines}.
\end{proposition}

\begin{proof}
First, assuming \(\alpha_2 > 0\), we rewrite~\eqref{eq:staircase-2lines} breaking for convenience its top line at the term where the second component of the weight is zero:
\begin{multline}
\label{eq:staircase-3lines-ogr2}
0 \to 
\wedge^{2n}V^\vee \otimes \Sigma^{\alpha_2-1,\alpha_1+1-2n}\cU_2^\vee \to
\dots \to
\wedge^{\alpha_1+2}V^\vee \otimes \Sigma^{\alpha_2-1,-1}\cU_2^\vee \\ \to
\wedge^{\alpha_1+1}V^\vee \otimes \Sigma^{\alpha_2-1,0}\cU_2^\vee \to
\dots \to
\wedge^{\alpha_1-\alpha_2+2}V^\vee \otimes \Sigma^{\alpha_2-1,\alpha_2-1}\cU_2^\vee \\ \to
\wedge^{\alpha_1-\alpha_2}V^\vee \otimes \Sigma^{\alpha_2,\alpha_2}\cU_2^\vee \to
\dots \to
\Sigma^{\alpha_1,\alpha_2}\cU_2^\vee \to 0.
\end{multline}
Applying the functor~\(\tilde\Phi_k\) and using Lemma~\ref{lem:tphi-cu} 
we see that each line of~\eqref{eq:staircase-3lines-ogr2} gives the corresponding line of~\eqref{eq:staircase-3-lines}
(the last~\(k - 2\) terms of the first line of~\eqref{eq:staircase-3lines-ogr2} are taken by~\(\tilde\Phi_k\) to zero, 
but the previous terms are shifted by~\(k - 2\) to the right).
By the uniqueness property of the staircase complex, the differentials in the obtained complex are the same as in~\eqref{eq:staircase-3-lines}.

The case~\(\alpha_2 = 0\) is analogous: we only need to compare the last lines of~\eqref{eq:staircase-3-lines} and~\eqref{eq:staircase-3lines-ogr2}.
\end{proof}

Recall the resolution~\eqref{eq:def-ckk} for the bundles~\(\cK_{\IGr(k,V)}^{\alpha_1,\alpha_2}\).

\begin{corollary}
\label{cor:phi-ck}
If~\(2n - 2 \ge \alpha_1 \ge \alpha_2 \ge 0\), then there is an isomorphism
\begin{equation}
	\Phi_k(\cK_{\IGr(2,V)}^{\alpha_1,\alpha_2}) \cong 
	\begin{cases}
		\cK_{\IGr(k,V)}^{\alpha_1,\alpha_2}, & \text{if \(\alpha_1 \le 2n-k\);}\\
		0, & \text{if \(2n-k+1 \le \alpha_1 \le 2n-2\).}
	\end{cases}
\end{equation}
\end{corollary}
\begin{proof}
First, assume~\(\alpha_1 \le 2n - k\).
Since~\(\cK_{\Gr(2,V)}^{\alpha_1,\alpha_2}\) is defined as a truncation of~\eqref{eq:staircase-3lines-ogr2}
and~\(\cK_{\Gr(k,V)}^{\alpha_1,\alpha_2}\) is defined as a truncation of its image~\eqref{eq:staircase-3-lines} under~\(\tilde\Phi_k\),
and the truncations match up, 
it follows that~\(\tilde\Phi_k(\cK_{\Gr(2,V)}^{\alpha_1,\alpha_2}) \cong \cK_{\Gr(k,V)}^{\alpha_1,\alpha_2}\).
Therefore, applying Lemma~\ref{lem:phi-tphi}, we deduce the claim.

Now, assume that~\(2n - k + 1 \le  \alpha_1 \le 2n - 2\).
Then we apply Lemma~\ref{lem:tphi-cu}  to the first row of~\eqref{eq:staircase-3lines-ogr2} 
and conclude that~\(\tilde\Phi_k(\cK_{\Gr(2,V)}^{\alpha_1,\alpha_2})\) is zero.
\end{proof}

\begin{proof}[of Theorem~\textup{\ref{thm:main}} for~\(k \ge 3\)]
We define the complex~\eqref{eq:ckk-special} by applying the functor~\(\Phi_k\) 
to~\(\cK_{t+k-2}^\bullet\) taken on~\(\IGr(2,V)\) and using Corollary~\ref{cor:phi-ck} to identify the terms. 
In particular, the last~\(k-2\) terms vanish, so the image is~\(\cK_{t}^\bullet\). 
It follows from the case~\(k = 2\) of Theorem~\ref{thm:main} 
that the obtained complex~\(\cK_{t}^\bullet\) is quasi-isomorphic to
\begin{equation*}
\cK_{t}^\bullet \cong
\begin{cases}
\Phi_k\left(\wedge_\Sp^{t+k-2}\cS_2(-1)\right), & \text{if~\(0 \le t \le n-2\)},\\
\Phi_k\left(\wedge_\Sp^{2n - 2 - (t+k-2)}\cS_2(-1)\right)[-1], & \text{if \(n \le t \le 2n-k\)},\\
0, & \text{otherwise}.
\end{cases}
\end{equation*}
Now we apply Lemma~\ref{lem:phi-cs} and obtain~\eqref{eq:ch-cckk}.
\end{proof}

\section{Bicomplexes}\label{sec:correct_maps}

In this section for any~\(0 \le t \le 2n-2\) we construct on~\(\IGr(2, V)\) a bicomplex
\begin{equation}
\label{eq:bicomplex}
\begin{tikzcd}[column sep = 15pt]
&&&&& \mathrm{S}^{t} \cU^\vee\\
&&&& \Sy^{t}\cU^\vee(-1) \rar[dotted] & V^\vee \otimes \Sy^{t-1} \cU^\vee \uar\\
&&& \Sigma^{i+1, i+1-t} \cU^\vee \rar[dotted] & \reflectbox{\(\ddots\)} \uar \rar[dotted] &  \vdots \uar \\
&& \Sigma^{i, i+t} \cU^\vee \rar[dotted] &V^\vee \otimes \Sigma^{i, i+1-t} \cU^\vee \uar \rar[dotted] & \reflectbox{\(\ddots\)} \uar \rar[dotted] &\vdots \uar \\
& \Sy^{t}\cU(1) \rar[dotted] & \reflectbox{\(\ddots\)}\uar \rar[dotted]&\reflectbox{\(\ddots\)} \uar \rar[dotted]&\reflectbox{\(\ddots\)} \rar[dotted] \uar &\wedge^{t-1} V^\vee \otimes \cU^\vee\uar\\
\Sy^{t} \cU \rar& V^\vee \otimes \Sy^{t-1} \cU  \uar \rar &\dots \rar \uar & \dots \rar \uar &\wedge^{t-1} V^\vee \otimes \cU \rar \uar & \wedge^{t} V^\vee \otimes \cO, \uar 
\end{tikzcd}
\end{equation}
whose columns are the right resolutions of~\(\cE^{t-b,b}\) from Theorem~\ref{thm:bicomplex},
and the horizontal arrows are chosen in such a way 
that the induced differentials on these bundles coincide with the differentials of the complex~\(\cE_{t}^\bullet\).
\begin{remark}\label{rem:bicomplex2}
	In particular, if~\(t = n-1\), then the totalization of the bicomplex is acyclic. This is the case studied in~\cite[Proposition~5.3]{lines}, see~Remark~\ref{rem:bicomplex}. It is claimed there that the horizontal maps are equal to~\(\vn_2\). However, as we prove further, they are given by a non-trivial linear combination of~\(\vn_1\) and~\(\vn_2\). So, since~\(\vn_1 \neq 0\) in general, and these morphisms are unique, as we discuss in Remark~\ref{rem:E^bullet_unique}, we get a contradiction.
\end{remark}

Recall the morphism~\(\vn = \tfrac1{b+1}\vn_1 + \vn_2 \colon \wedge^{a} V^\vee \otimes \Sy^{b} \cU \to \wedge^{a + 1} V^\vee \otimes \Sy^{b - 1} \cU\)
defined in~\eqref{eq:def-v1}--\eqref{eq:def-v2}:
\begin{equation*}
\vn(\lambda \otimes P) = 
\tfrac1{b+1}(\lambda_1 \wedge \omega \otimes P_1 + 
\lambda_2 \wedge \omega \otimes P_2) + 
\lambda \wedge (\omega_1 \otimes P_1 + \omega_2 \otimes P_2).
\end{equation*}
We checked in Proposition~\ref{prop:ck-complex} that it induces a complex on the subbundles~\(\cE^{a,b} \subset \wedge^{a} V^\vee \otimes \Sy^{b} \cU\).
Now we check that it also induces a complex of the ambient bundles.

\begin{lemma}
The composition
\begin{equation*}
\wedge^{a} V^\vee \otimes \Sy^{b} \cU \xrightarrow{\ \vn_1 + (b + 1)\vn_2\ }
\wedge^{a + 1} V^\vee \otimes \Sy^{b - 1} \cU \xrightarrow{\ \vn_1 + b\vn_2\ }
\wedge^{a + 2} V^\vee \otimes \Sy^{b - 2} \cU
\end{equation*}
is zero.
\end{lemma}

\begin{proof}
The composition of these arrows is the sum of four maps:~\(b(b+1)\vn_2^2\), \(\vn_1^2\), 
\(b \, \vn_2 \circ \vn_1\), and~\((b+1)\vn_1 \circ \vn_2\). 
These compositions take~\(\lambda \otimes P\) to
\begin{align*}
b(b + 1)\lambda \wedge (\omega_1 \wedge \omega_2 \otimes P_{12} + \omega_2 \wedge \omega_1 \otimes P_{12});
\\[1ex]
\lambda_1 \wedge \omega_1 \wedge \omega \otimes P_{11} 
+ \lambda_2 \wedge \omega_1 \wedge \omega \otimes P_{12}\hphantom{;}
\\
+ \lambda_1 \wedge \omega_2 \wedge \omega \otimes P_{12} 
+ \lambda_2 \wedge \omega_2 \wedge \omega \otimes P_{22};
\\[1ex]
b\lambda_1 \wedge \omega \wedge (\omega_1 \otimes P_{11} + \omega_2 \otimes P_{12})\hphantom{;} 
\\
+ b\lambda_2 \wedge \omega \wedge (\omega_1 \otimes P_{12} + \omega_2 \otimes P_{22});
\\[1ex]
- (b + 1)\lambda_1 \wedge \omega \wedge (\omega_1 \otimes P_{11} + \omega_2 \otimes P_{12})\hphantom{;}
\\
- (b + 1)\lambda_2 \wedge \omega \wedge (\omega_1 \otimes P_{12} + \omega_2 \otimes P_{22}),
\end{align*}
where as usual we choose a basis~\(e_1,e_2\) in~\(U\) and write~\(\lambda_i\) and~\(\omega_i\) 
for the evaluation of the corresponding skew-form on~\(e_i\) 
and~\(P_{ij}\) for the second derivatives of~\(P\) in the directions~\(e_i\) and~\(e_j\).
It is easy to see that the first summand is zero and the last three summands cancel out.
\end{proof}

Consider also the unique~\(\GL(V)\)-equivariant morphism
\begin{equation*}
	\vn_0 \colon \wedge^{a} V^\vee \otimes \Sy^{b} \cU \to \wedge^{a - 1} V^\vee \otimes \Sy^{b + 1} \cU(1).
\end{equation*} 
It is, up to twist, the differential of Koszul complex~\eqref{eq:koszul-right} and given by the following composition of canonical morphisms:
\begin{multline*}
	\wedge^{a} V^\vee \otimes \Sy^{b} \cU \to 
	\wedge^{a-1} V^\vee \otimes V^\vee \otimes \Sy^{b} \cU 
	\to \wedge^{a-1} V^\vee \otimes \cU^\vee \otimes \Sy^{b} \cU \\
	\to \wedge^{a-1} V^\vee \otimes \Sigma^{b,-1} \cU 
	\cong \wedge^{a - 1} V^\vee \otimes \Sy^{b + 1} \cU(1),
\end{multline*}
At the point~\([U]\) with a basis~\(e_1, e_2\) of~\(U\) it acts by
\begin{equation}
\label{eq:def-vn0}
\vn_0(\lambda \otimes P) = (\lambda_1 \otimes e_2 \cdot P - \lambda_2 \otimes e_1 \cdot P) \cdot \det U^\vee,
\end{equation}
since the composition of the first three maps is~\(\sum \partial/\partial e_i \otimes e_i\) and the last isomorphism is~\(e_i \mapsto e_j \cdot \det U^\vee\) for~\(\{i,j\}=\{1,2\}\).

\begin{proposition}
\label{prop:bicomplex-appendix}
For any~\(a,b \ge 1\) the diagram 
\begin{equation*}
\xymatrix@C=8em{
\wedge^{a} V^\vee \otimes \Sy^{b} \cU \ar[r]^{(b + 2)(\vn_1 + (b+1)\vn_2)} \ar[d]_{\vn_0} & 
\wedge^{a + 1} V^\vee \otimes \Sy^{b - 1} \cU \ar[d]^{\vn_0}
\\
\wedge^{a - 1} V^\vee \otimes \Sy^{b + 1} \cU(1) \ar[r]^{b(\vn_1 + (b+2)\vn_2)} & 
\wedge^{a} V^\vee \otimes \Sy^{b} \cU(1) 
}
\end{equation*}
is anticommutative.
\end{proposition}

\begin{proof}
Using formulas~\eqref{eq:def-v1}, \eqref{eq:def-v2}, and~\eqref{eq:def-vn0} it is straightforward to check
that the composition~\mbox{\((\vn_1 + (b + 2)\vn_2) \circ \vn_0\)} takes~\(\lambda \otimes P\) to
\begin{align*}
& (\lambda_{12} \wedge \omega \otimes (P + e_1 \cdot P_1) 
+ \lambda_{12} \wedge \omega \otimes (P + e_2 \cdot P_2)
\\
+ &(b+2) \lambda_1 \wedge (\omega_1 \otimes e_2 \cdot P_1 + \omega_2 \otimes (P + e_2 \cdot P_2)) 
\\
- &(b+2) \lambda_2 \wedge (\omega_1 \otimes (P + e_1 \cdot P_1) + \omega_2 \otimes e_1 \cdot P_2)) \cdot \det U^\vee,
\end{align*}
while the composition~\(\vn_0 \circ (\vn_1 + (b + 1)\vn_2)\) takes~\(\lambda \otimes P\) to
\begin{align*}
& (\lambda_1 \wedge \omega_1 \otimes e_2 \cdot P_1 - \lambda_{12} \wedge \omega \otimes e_2 \cdot P_2 + \lambda_2 \wedge \omega_1 \otimes e_2 \cdot P_2
\\
- & \lambda_{12} \wedge \omega \otimes e_1 \cdot P_1 - \lambda_1 \wedge \omega_2 \otimes e_1 \cdot P_1 - \lambda_2 \wedge \omega_2 \otimes e_1 \cdot P_2
\\
- & (b + 1) \lambda_1 \wedge (\omega_1 \otimes e_2 \cdot P_1 + \omega_2 \otimes e_2 \cdot P_2)
\\
+ & (b + 1) \lambda_2 \wedge (\omega_1 \otimes e_1 \cdot P_1 + \omega_2 \otimes e_1 \cdot P_2)) \cdot \det U^\vee.
\end{align*}
It remains to note that the sum of the first expression multiplied by~\(b\)
and the second expression multiplied by~\(b + 2\) is zero.
\end{proof}

\begin{corollary}
\label{cor:bicomplex-maps}
The total complex of the diagram~\eqref{eq:bicomplex}, where the vertical arrows are given by~\(\vn_0\)
and the horizontal arrow
\begin{equation*}
\wedge^{t-b-c}V^\vee \otimes \Sigma^{c,-b}\cU^\vee \to \wedge^{t-b-c+1}V^\vee \otimes \Sigma^{c,1-b}\cU^\vee 
\end{equation*}
is given by the map~\((-1)^c\left(\tfrac{b}{(b+c)(b+c+1)}\vn_1 + \tfrac{b}{b+c}\vn_2\right)\),
is quasi-isomorphic to~\eqref{eq:ck-special}.
In particular, if~\(t = n -1\) the total complex of the diagram~\eqref{eq:bicomplex} is acyclic.
\end{corollary}

\printbibliography

\textsc{HSE University, Russian Federation}

\end{document}